\documentclass[a4paper,12pt]{amsart}
\usepackage{amsfonts}
\usepackage{amsmath}
\usepackage{amssymb}
\usepackage{lipsum}
\usepackage{mathrsfs}
\usepackage{xcolor}
\usepackage{hyperref}
\usepackage{amsmath}
\usepackage{amssymb}
\usepackage{tabularx}
\usepackage{orcidlink}
\usepackage{enumerate}
\usepackage[utf8]{inputenc}
\numberwithin{equation}{section}

\usepackage{graphicx}
\usepackage{enumitem}

\def\R2n{{\mathbb R}^{2n}}

\def\R2{{\mathbb R}^2}
\def\R2n{{\mathbb R}^{2n}}

\begin{document}
	\numberwithin{figure}{section}
	\numberwithin{table}{section}
	\numberwithin{equation}{section}
	\newtheorem{Theorem}{Theorem}
	\newtheorem{corollary}{Corollary}[Theorem]
	\newtheorem{conjecture}{Conjecture}[Theorem]
	\newtheorem{proposition}{Proposition}[section]
	\newtheorem{Lemma}{Lemma}[section]
	\theoremstyle{definition}
	\newtheorem{conj}{Conjecture}[section]
	\newtheorem{exmp}{Example}[section]
	\newtheorem{exmp*}{Example}
	\theoremstyle{remark}
\title{
Power Integral Bases in Polynomial Compositions}

\author[A. Choudhary]{Aakash Choudhary} \address[A. Choudhary]{Department of Mathematics, Indian Institute of Science Education and Research, Pune, Maharashtra,
India-411008}\email{achoudhary1396@gmail.com}

\author[S. Pisolkar]{Supriya Pisolkar\,
}\address[S. Pisolkar]{Department of Mathematics, Indian Institute of Science Education and Research, Pune, Maharashtra,
India-411008}\email{supriya@iiserpune.ac.in}

\author[P. Yadav]{Prabhakar Yadav} \address[P. Yadav]{Ashoka University, Sonepat, Haryana, India-131021}\email{pkyadav914@gmail.com}
\begin{abstract}

In this paper, we study the monogeneity of a special class of composed polynomials of the form 
$
(f \circ g)(x) = (x^m + c)^n + a(x^m + c)^{n-1} + d(x^m + c)^{n-2} + b,$
where \( f(x) = x^n + a x^{n-1} + d x^{n-2} + b \in \mathbb{Z}[x] \) satisfies \( a^2 = 4d \) and \( g(x) = x^m + c \in \mathbb{Z}[x] \). 
Assuming that \( (f \circ g)(x) \) is irreducible over \( \mathbb{Q} \), we obtain necessary and sufficient conditions on the parameters \( a, b, c, d, m, n \) for the polynomial to be monogenic. 
These conditions help to identify when the set \( \{1, \theta, \dots, \theta^{mn-1}\} \) forms an integral basis of the number field \( \mathbb{Q}(\theta) \), where \( \theta \) is a root of \( (f \circ g)(x) \). We also provide lower bound for the counting of such monogenic polynomials. Furthermore, we study the behaviour of solutions to certain related differential equations and present a class of polynomials with non-square-free discriminants as an application of the main results.

\end{abstract}
\maketitle

\par algebraic integer; Power basis.

\section{Introduction}  

Let \( K \) be an algebraic number field, \(\mathcal{O}_K \) denotes the ring of integers and \( \theta \in \mathcal{O}_K \).
A fundamental question in algebraic number theory is whether \( K \) is monogenic,  
that is, whether there exists an element \( \theta \in \mathcal{O}_K \) such that  
$
\mathcal{O}_K = \mathbb{Z}[\theta],$
or equivalently, whether \( \mathcal{O}_K \) possesses a integral basis generated by powers of $\theta$. The problem of providing an arithmetic characterization of monogenic number fields was first posed by Hasse~\cite{hasse} in 1960. 
Although a complete classification remains tricky, substantial progress has been achieved in this area. 
Notably, Ga\'{a}l~\cite{Gaal_book} made important contributions by developing methods based on index form equations 
to classify monogenity in number fields of lower degrees.

\vspace{2mm}
Let \( K = \mathbb{Q}(\theta) \) be an algebraic number field, where \( \theta \in \mathcal{O}_K \).  Suppose \( h(x) \) is the minimal polynomial of \(\theta\) over \( \mathbb{Q} \), having degree \( n \).  
Let \( d_K \) represent the discriminant of the field \( K \), and \( D_h \) the discriminant of the polynomial \( h(x) \).  
It is well known that these two discriminants are connected by the relation:
$$
D_h = [\mathcal{O}_K : \mathbb{Z}[\theta]]^2d_K.
$$
A polynomial \( h(x) \) is said to be monogenic if \( \mathcal{O}_K = \mathbb{Z}[\theta] \), 
which is equivalent to the equality \( D_h = d_K \).  
In this case, the set \(\{1, \theta, \dots, \theta^{n-1}\}\) forms an integral basis of \( K \), 
and hence \( K \) is called a monogenic number field.  
More generally, a number field \( K \) is said to be monogenic if there exists an element 
\(\theta \in \mathcal{O}_K\) such that \( \mathcal{O}_K = \mathbb{Z}[\theta] \).  Note that if a polynomial \( h(x) \) is monogenic, then the number field 
\( K = \mathbb{Q}(\theta) \), where \(\theta\) is a root of \( h(x)\), is also monogenic.  
However, the converse does not necessarily hold.  
For example, let \(\varepsilon\) and \(\gamma\) be the roots of \( h_1(x) = x^2 - 5 \) 
and \( h_2(x) = x^2 - x - 1 \), respectively.  
Although \( \mathbb{Q}(\varepsilon) = \mathbb{Q}(\gamma) \), 
the polynomial \( h_2(x) \) is monogenic while \( h_1(x) \) is not.

\vspace{2mm}
In 1878, Dedekind provided a necessary and sufficient criterion for the minimal polynomial \( f(x) \) of \(\theta\) such that a prime \( p \) does not divide \([\mathcal{O}_K : \mathbb{Z}[\theta]]\) (cf. \cite[Theorem 6.1.4]{h_cohen}, \cite{dedekind}). 
In 2016, Jhorar and Khanduja provided necessary and sufficient conditions for \( \mathcal{O}_K = \mathbb{Z}[\theta] \) when \(\theta\) is a root of an irreducible binomial \( x^n - b \in \mathbb{Z}[x] \) (cf. \cite{jhorar20166}). Later in 2016, Jakhar et.al. in \cite{jakhar_jnt} extended this result to the polynomial \( x^n + ax^m + b \in \mathbb{Z}[x] \). Gaál’s recent survey \cite{gaalsurvey} provides an overview of recent developments on monogenic number fields and power integral bases, including pure fields, trinomials, relative extensions, and algorithmic methods.


Recent work has shown growing interest in the monogeneity of power-compositional polynomials. Harrington and Jones \cite{harrington2022} studied power-compositional trinomials of the form \(F(x)=x^m+Ax^{m-1}+B\) and obtained conditions under which \(F(x^{p^n})\) is monogenic for every integer \(n\ge 0\) and a fixed prime \(p\). Jones \cite{jonespower} constructed infinite families of monic Eisenstein polynomials whose power-compositional polynomials remain monogenic. He also established a criterion for the monogenicity of power-compositional characteristic polynomials in terms of the period of the associated linear recurrence sequence \cite{jones2024monogenicity}. In a related direction, Kaur \textit{et al.} \cite{suman_ffa} gave necessary and sufficient conditions for the monogenity of monic irreducible power-compositional polynomials \(f(x^k)\in\mathbb{Z}[x]\).

\vspace{2mm}
 In 2020, Jones and Harrington \cite{jones_taiwan} examined several pairs of binomials \( f(x) = x^n - a \) and \( g(x) = x^m - b \) that possess the property where both \( f(x) \) and \( f(g(x)) \) are monogenic.  
In 2022, Ga{\'a}l \cite{gaal2022} described the monogenic properties of a class of binomial compositions of degree six. Recently, Jakhar \textit{et al.} \cite{jakhar2025study} extended it to any  $n,m \in \mathbb{N},$ and Sharma \textit{et al.} \cite{himanshu_ramanujan} investigated the monogenity of 
irreducible compositions \(f \circ g(x)\), where \(f(x) = x^n + ax + b \in \mathbb{Z}[x]\) 
and \(g(x) \in \mathbb{Z}[x]\) satisfies certain restrictions.

\vspace{2mm}
In this paper, we study the monogeneity of the composition of the quadrinomial \( f(x) = x^n + ax^{n-1} +dx^{n-2} + b \in \mathbb{Z}[x] \), where $a^2 =4d$ with \( n \geq 3 \), and the binomial \( g(x) = x^m + c \in \mathbb{Z}[x] \). The monogeneity of the quadrinomial $f(x)$ was explored by Jakhar in \cite{jakharrocky}.
It is important to observe that when $m=1$, the composition $f \circ g$ reduces to a mere translation of $f$ by an integer constant $c$. In this situation, the monogeneity of $f(x+c)$ and $f(x)$ are same. Hence, in what follows, we will discuss the necessary and sufficient conditions for the monogenity of $f(g(x))$  for $m \geq 2$.

\vspace{2mm}
Let \( K = \mathbb{Q}(\theta) \) be an algebraic number field, where \(\theta\) is a root of the irreducible polynomial  
$F(x) = (f \circ g)(x) = (x^m + c)^n + a(x^m + c)^{n-1} + d(x^m + c)^{n-2} + b$ 
over \(\mathbb{Q}\), where  \( a^2 = 4d \), \( n \geq 3 \) and \( m \geq 2 \).  
In this article, we apply Dedekind’s Criterion to establish necessary and sufficient conditions involving 
\( a, b, c, d, m, \) and \( n \) for the polynomial \( F(x) \) to be monogenic. 
Furthermore, we identify all primes that divide the index \([\mathcal{O}_K : \mathbb{Z}[\theta]]\). 
In addition, we illustrate our results through explicit examples. 
In certain cases of non-monogenic compositions, we also compute the index \( [\mathcal{O}_K : \mathbb{Z}[\theta]] \).

\vspace{2mm}
Let $D_f$ denote the discriminant of the polynomial  
\(
f(x) = x^n + ax^{n-1} + dx^{n-2} + b.
\)  
To compute $D_f$, note that under the condition $a^2 = 4d$, we obtain  $
f'(x) = x^{n-3}(nx^2 + a(n-1)x + (a/2)^2(n-2)).$
The derivative $f'(x)$ has rational roots given by 
\[
\varepsilon_1 = \varepsilon_2 = \cdots = \varepsilon_{n-3} = 0, \quad 
\varepsilon_{n-2} = -\tfrac{a}{2}, \quad 
\varepsilon_{n-1} = -\tfrac{a}{2} + \tfrac{a}{n}.
\]  
Therefore,  
\[
D_f = (-1)^{n(n-1)/2}\,\mathrm{Res}(f,f') 
= (-1)^{n(n-1)/2}\, n^n f(0)^{\,n-3} f(-a/2)\, f\!\left(-\tfrac{a}{2} + \tfrac{a}{n}\right).
\]
A straightforward computation leads to  
\begin{equation}{\label{22}}
D_f = (-1)^{n(n-1)/2}(-b)^{n-2}\{b(-n)^n+4 (n-2)^{n-2}(a/2)^n \}.
\end{equation}

\noindent Let $D_F$ denote the discriminant of the composition $F= f\circ g$. 
By \cite{jc}, we have

\begin{align*}
    D_F &= (-1)^{nm(3nm+n-2m-2)/2} (D_f)^m \operatorname{Res}(f\circ g,g^\prime) \\
    & = (-1)^{nm(3nm-2m+n-2)/2} (D_f)^m m^{mn} (-1)^{nm(m-1)} \prod_{i=1}^{m-1}f(g(0)) \\ 
    &=(-1)^{nm(3nm-2m+n-2)/2 + nm(m-1)} (D_f)^m m^{mn} f(c)^{m-1}. 
\end{align*}

\noindent
Combining this with eq. \ref{22} and simplifying the power of $-1$, we have the following value of discriminant of the composition $F=f\circ g$:
\begin{align}{\label{disc}}
    D_F = (-1)^{nm(nm-2m+n)/2}(-b)^{m(n-2)}[b(-n)^n+4 (n-2)^{n-2}(a/2)^n  ]^m m^{mn}f(c)^{m-1},
\end{align}
where $f(c) = c^n + ac^{n-1}+dc^{n-2} + b.$
\\

Specifically, we establish the following result.
\begin{Theorem}{\label{main thm}}
    Let $K = Q(\theta)$ be an algebraic number field with $\theta$ in the ring $\mathcal{O}_K$ of
algebraic integers of $K$ having minimal polynomial $F(x) = f\circ g (x)=  (x^m +c)^n + a(x^m +c)^{n-1} +d(x^m +c)^{n-2}+b$
over $\mathbb{Q}$, where $a^2=4d$, $n\geq 3,\; m\geq 2$.
A prime divisor of
the discriminant $D_F$ of $F(x)$ does not divide $[\mathcal{O}_K : \mathbb{Z}[\theta]]$ if and only if $p$ satisfies one of the
following conditions:
\begin{enumerate}
    \item[(i)] when $  p\mid b$ then $p^2\nmid b$. 
     \item [(ii)]
    when $p \nmid b$ and $p\mid m$ with $m=p^j s; j\geq 1, p\nmid s$, then the polynomials $(x^s +c)^n +a(x^s +c)^{n-1} +d(x^s+c)^{n-2} +b$ and $\frac{1}{p} n(x^s+c)^{p^j (n-1)}t(x) + \frac{1}{p} \{(-1)^{p^j} (a(x^s+c)^{n-1} + d (x^s+c)^{n-2} + b)^{p^j}  + a(x^m+c)^{n-1} +d(x^m+c)^{n-2} +b \}$ are co-prime modulo $p$ where $t(x)= \sum_{i=1}^{p^j}{{p^j}\choose {i}} (x^s+c)^{p^j-i} (-c)^i +c$.
\item[(iii)] when $p\nmid bm$ and $p\mid d$ with $n=sp^k$,  $p \nmid s $ then 
\begin{enumerate}
    \item [(a)] if $k\geq 1$ then either $p\mid a_1$, $p\nmid b_1$ or $p\nmid a_1((-b_1)^n - a_1^n(-b)^{n-1})$, where $ b_1 = \frac{b+(-b)^{p^k}}{p}$ and $a_1 = \frac{a}{p}$.
    \item [(b)] if $k=0$ then $p^2 \nmid  f(c)$.
\end{enumerate}

\item[(iv)]  when  $p=2$ and $2\nmid dbm$ then $b \equiv 1(mod\; 4)$
\item[(v)] when $p\nmid abm$ 
then \( p^2 \nmid  f(c) \) and \( p^2 \nmid  \{(-n)^nb +4 (n-2)^{n-2}(a/2)^n\} \).
\end{enumerate}
\end{Theorem}

The corollary below directly follows from the theorem stated above.

\begin{corollary}{\label{c1}}
Let \( K = \mathbb{Q}(\theta) \), and consider the polynomial  
$
F(x) = (x^m + c)^n + a(x^m + c)^{n-1} + d(x^m + c)^{n-2} + b,$
as defined in Theorem \ref{main thm}.  
Then the ring of integers \( \mathcal{O}_K \) equals \( \mathbb{Z}[\theta] \) 
if and only if every prime \( p \) dividing the discriminant \( D_F \) 
satisfies one of the conditions \( (i) \text{--} (v) \) listed in Theorem \ref{main thm}.
\end{corollary}

\begin{corollary}{\label{c2}}
Let \( K = \mathbb{Q}(\theta) \), and let the polynomial \( F(x) \) be as given in Theorem \ref{main thm}.  
Assume that \( \operatorname{rad}(m) \mid b \), \( \gcd(d, n) = 1 \), and \( b \equiv 1 \pmod{4} \).  
Then the ring of integers \( \mathcal{O}_K \) equal to \( \mathbb{Z}[\theta] \) if and only if every prime \( p \) dividing the discriminant \( D_F \) satisfies one of the following conditions:  

(i) \( p \mid b \) and \( p^2 \nmid b \); or  

(ii) \( p \nmid b \), and both 
\( p^2 \nmid \{(-n)^n b + 4 (n-2)^{n-2}(a/2)^n\} \) 
and \( p^2 \nmid f(c) \).
\end{corollary}

Consider the polynomial \( F(x)\) and the assumptions as in corollary \ref{c2}. 
Let \( q \) be a prime dividing \(\gcd(b,c) \) such that \( q^2 \nmid b \). For \( j \geq 0 \), define \( F_{q,j} = F(x^{q^j}) \). Then, for all \( j \geq 0 \), the polynomial \( F_{q,j} \) is \( q \)-Eisenstein and
the discriminant $D_{F_{q,j}}$ of the polynomial $F_{q,j}$ satisfies 
$$
|D_{F_{q,j}}| = \bigg|[(-b)^{n-2}\{(-n)^nb +4 (n-2)^{n-2}(a/2)^n\} ]^{mq^j} (q^jm)^{mnq^j} f(c)^{mq^j -1} \bigg|.
$$
\noindent
By employing the formula for \( |D_{F_{q,j}}| \) together with Corollary \ref{c2} and Theorem \ref{main thm}, 
we deduce that the composition of a certain class of polynomials yields a monogenic polynomial.
\begin{corollary}{\label{c3}}
Consider the polynomial \( F(x) \) and the assumptions stated in Corollary \ref{c2}.  
Let \( q \) be a prime dividing \( \gcd(b, c) \) such that \( q^2 \nmid b \).  
Then the polynomial \( F_{q,j} \) (as defined above) is \( q \)-Eisenstein for all \( j \geq 0 \).  
Moreover, \( F_{q,j}(x) \) is monogenic for every \( j \geq 0 \) if and only if each prime \( p \) dividing \( D_F \) 
satisfies one of the following conditions:  

(i) \( p \mid b \) and \( p^2 \nmid b \); or  

(ii) \( p \nmid b \), and both  
\( p^2 \nmid \{(-n)^n b + 4 (n-2)^{n-2}(a/2)^n\} \)  
and \( p^2 \nmid f(c) \).
\end{corollary}

\begin{corollary}{\label{analytic c}}
 Let $f(x) = x^n+ax^{n-1}+d x^{n-2}+b $ and be such that \(n\ge 3\). Let \(K=\mathbb{Q}(\theta)\) with \(\theta\) having the minimal polynomial
\[
F(x)=(x^m+c)^n+a(x^m+c)^{n-1}+d (x^m +c)^{n-2}+ b\in \mathbb{Z}[x]  
\]
with $a^2 =4d$, \(m\ge 2\) and \(\operatorname{rad}(m)\mid b\). Suppose \(f(x)\) and \(F(x)\) are irreducible. Then \(f(x)\) and \(F(x)\) are monogenic if

\begin{enumerate}
    \item 
$
\text{If } b \text{ is square-free, then for every prime } p\nmid b,
$
$
\nu_p\left((-n)^n b+4(n-2)^{n-2}\left(\frac{a}{2}\right)^n\right)\leq 1$ and 
$\nu_p(f(c))\leq 1$.
    
    \item If $2\nmid dm$ then $b\equiv 1 \pmod 4$.
    \item For each prime $p$ such that $p\mid \gcd(d,n)$, we have, 
$(-b)^p \not\equiv -b \pmod{p^2}$
 when $p^2 \mid a$, and $(-b)^p \equiv -b \pmod {p^2}$ when $p \mid\mid a$.
\end{enumerate}

\end{corollary}

\begin{Theorem}{\label{t2}}
    Let \begin{align}{\label{diff}}
    \bigg( \frac{d^m }{dx^m} +c \bigg )^n y+ a\bigg( \frac{d^m }{dx^m} +c \bigg )^{n-1} y+d\bigg( \frac{d^m }{dx^m} +c \bigg )^
    {n-2}y+by=0 
    \end{align}
    be a differential equation with $a^2=4d$, where $n\geq 3,\; m\geq 2$. Let $\mathcal{F} =  (z^m + c)^n + a(z^m + c)^{n-1}+ d(z^m + c)^{n-2} + b$ be the irreducible auxiliary equation of \eqref{diff}
  with a root $\theta$ such that $K(\theta)$ is the splitting field of $\mathcal{F}$. If for each prime $p$ dividing the discriminant $D_{\mathcal{F}}$
 of $\mathcal{F}(z)$ satisfies any one of the conditions (i) to (v) of Theorem \ref{main thm}, then the general
 solution of the given differential equation \eqref{diff} is of the form
\begin{align*}
    y(x) = \sum_{i=1}^{mn} \alpha_i\prod_{j=1}^{mn}e^{c_{j-1}^{(i)} \theta^{j-1}x},
\end{align*}
where $c_{j-1}^{(i)}$ are integers and $\alpha_i$ are arbitrary real constants, for all $1\leq i,j \leq mn$.
\end{Theorem}

In recent years, several quantitative results on monogenic polynomials have been obtained in the literature. Kedlaya \cite{Kedlaya} constructed infinitely many monogenic polynomials of every degree greater than one with square-free discriminant. Later, Jones \cite{JonesNonSquarefree} produced infinite families of monogenic polynomials with non-square-free discriminants, and Barman, Narode and Wagh \cite{BarmanNarodeWagh} established further such families by related methods. Motivated by these developments, we provide a quantitative consequence of Corollary \ref{analytic c}.  More precisely, under the assumption of the abc-conjecture, we obtain a lower bound for the number of pairs of polynomials
\[
f(x)=x^{n}+ax^{n-1}+dx^{n-2}+b
\quad \text{and} \quad
F(x)=(x^{m}+c)^{n}+a(x^{m}+c)^{n-1}+d(x^{m}+c)^{n-2}+b
\]
for which both \(f(x)\) and \(F(x)\) are monogenic.

\vspace{2mm}

Suppose $d\neq \pm 1$ and let $\ell$ be a prime divisor of $d$ such that $\ell\nmid n$, and define $\rho=\operatorname{rad}(\gcd(d,n)), ~
\kappa=\operatorname{rad}(m\ell).$
Assume that $\gcd(\kappa,\rho)=1$, and set
\[
\delta=
\begin{cases}
1, & \text{if }2\mid dm,\\[1mm]
4, & \text{if }2\nmid dm.
\end{cases}
\]
Further, let $ G(y)=(-n)^n y+4(n-2)^{n-2}\left(\frac{a}{2}\right)^n, $
and define
$
D_G=\gcd\{G(t):t\in \mathbb Z\}.
$
Let $D_G'$ be the smallest divisor of $D_G$ such that $D_G/D_G'$ is square-free. For
each prime $p$, write
$
q_p=\nu_p(D_G'),
$
and let $\omega_G(p)$ denote the number of integers $\alpha$, with
$1\le \alpha\le p^{2+q_p}$, such that
$
G(\alpha)\equiv 0 \pmod{p^2}.
$
Also, define
\[
\Lambda_n=
\prod_{p\le \sqrt n}\frac{1}{p^2}
\prod_{p>\sqrt n}\left(1-\frac{n}{p^2}\right),
\qquad
\lambda_G=
\prod_{p\ \mathrm{prime}}
\left(1-\frac{\omega_G(p)}{p^{2+q_p}}\right).
\]

\begin{Theorem}{\label{analytic t}}
Let $B,C >0$ be integers. 
With the above notation, 
if the abc-conjecture is true, then there are at least
\[
\Lambda_n\lambda_G
\frac{\varphi(\kappa)}{\delta\,\ell\,\kappa^2\zeta(2)\rho^2}\,BC+o(BC)
\]
pairs $(b,c)$ with $1\le b\le B$ and $1\le c\le C$ such that both
$
f(x)=x^n+ax^{n-1}+dx^{n-2}+b
$
and 
$
F(x)=(x^m+c)^n+a(x^m+c)^{n-1}+d(x^m+c)^{n-2}+b
$
are irreducible and monogenic.
\end{Theorem}

We remark that the abc-conjecture is utilized specifically to derive the asymptotic distribution of $c$. For the weaker requirement of proving that infinitely many such $c$ exist along with the asymptotic formula for $b$, one may rely instead on the unconditional results of Erd\"{o}s \cite{Erdos53} concerning the squarefree values of irreducible polynomials.

Now we present an instance that illustrates the application of above  Corollaries and Theorem \ref{t2}, where \( K = Q(\theta) \), and \( \theta \) denotes a root of the polynomial \( F(x) \).

\begin{exmp}

\begin{enumerate}
\item  Consider the differential equation
\begin{align}{\label{e1}}
          \bigg( \frac{d^5 }{dx^5} +5 \bigg )^p y+ 10\bigg( \frac{d^5 }{dx^5} +5 \bigg )^{p-1} y+25\bigg( \frac{d^5 }{dx^5} +5 \bigg )^
    {p-2}y+5y=0 
    \end{align}
    with auxiliary equation \( F(x) = (x^5 + 5)^p + 10(x^5 + 5)^{p-1} +25(x^5+5)^{p-2} + 5 \), where $p\neq 5$ is an odd prime. 
Note that \( F(x) \) is a \( 5 \)-Eisenstein polynomial. We apply Corollary \ref{c2} with \( b=m=c= 5, a=10, d=25\) and \( n=p \). Observe that \( 5 \) divides \( b \) but \( 5^2 \nmid b\),  therefore \( \mathcal{O}_K = \mathbb{Z}[\theta] \) if and only if 
\((-p)^{p}+4 (p-2)^{p-2}5^{p-1}\) and \(   4\cdot5^{p-1}+1 \)
are both square-free. It can be verified that \((-p)^{p}+4 (p-2)^{p-2}5^{p-1}\) and \(   4\cdot5^{p-1}+1 \) are both square-free for all primes \( p < 67 \), except for \( p = 13 \). Hence, \( \mathcal{O}_K = \mathbb{Z}[\theta] \) holds for these primes. Therefore, using Theorem \ref{t2}, the differential equation \eqref{e1} has general solution of the form 
\begin{align*}
    y(x) = \sum_{i=1}^{5p} s_i\prod_{j=1}^{5p}e^{r_{j-1}^{(i)} \theta^{j-1}x},
\end{align*}
where $r_{j-1}^{(i)}$ are integers and $s_i$ are arbitrary real constants, for all $1\leq i,j \leq 5p$. \\

\item \( F(x) = (x^5 + 5)^p + 10(x^5 + 5)^{p-1} +25(x^5+5)^{p-2} + 5 \), where $p\neq 5$ is an odd prime from part (1). Keeping in mind that $F(x)$ is $5$-Eisenstien and applying Corollary \ref{c3}, we conclude that the polynomial \( F_{q,j}(x) \) is monogenic for all \( j \geq 0 \) and primes \( p < 67 \), excluding \( p = 13 \).

\vspace{2mm}
\noindent
When \( p = 13 \), it can be verified that \( |D_F| = 5^{129}\cdot 17^{8} \cdot u^4 \cdot v^5 \), where \( u ,v \) are square-free integers. Observing that \( D_F = [\mathcal{O}_K : \mathbb{Z}[\theta]]^2 d_K \), it follows from Theorem \ref{main thm} (v) that the integers $5,u$ and $v$ does not divide \( [\mathcal{O}_K : \mathbb{Z}[\theta]] \),  and that \( 17 \) divide \( [\mathcal{O}_K : \mathbb{Z}[\theta]] \). Hence \( [\mathcal{O}_K : \mathbb{Z}[\theta]] >1 \). 
Therefore for $p = 13$, using Corollary \ref{c3}, we conclude that $F_{q,j}(x)$ is non-monogenic for all $j \geq 0$.
\end{enumerate}
\end{exmp}
Here, all calculations to verify the factorization of the discriminant were conducted using SageMath \cite{sage}.

\section{Preliminaries}
In the following discussion, let \( p \) denote a prime number.  
For a polynomial \( h(x) \in \mathbb{Z}[x] \), we write \( \overline{h}(x) \) for the polynomial over \( \mathbb{Z}/p\mathbb{Z} \) 
obtained by reducing each coefficient of \( h(x) \) modulo \( p \).

\vspace{2mm}

The following result, referred to as Dedekind's criterion, will be used in the subsequent discussion. Dedekind (see \cite[Theorem 6.1.4]{h_cohen}, \cite{dedekind}) established the equivalence of assertions (i) and (ii). K. Uchida \cite{uchida} proved the equivalence of (i) and (iii), and a simple proof of the equivalence between (ii) and (iii) is given in \cite[Lemma 2.1]{jakhar_jnt}.

\begin{Lemma}{\cite {h_cohen,dedekind, uchida, jakhar_jnt}}{\label{jones}}
    Let $f(x) \in \mathbb{Z}[x]$ be a monic irreducible polynomial having the factorization $h_1(x)^{e_1}\cdots h_t(x)^{e_t}$ modulo  prime $p$ as a product of powers of distinct irreducible
polynomials over $\mathbb{Z}/p\mathbb{Z}$ with $h_i(x) \in  \mathbb{Z}[x]$ monic. Let $K = Q(\theta)$ with $\theta$ a root of $f(x)$.
Then the following statements are equivalent
\begin{enumerate}
    \item[(i)]  $p$ does not divide $[\mathcal{O}_K :\mathbb{Z}[\theta]]$.
    \item[(ii)] For each $i$, we have either $e_i = 1$ or $\overline{h_i}(x)$ does not divide $\overline{M}(x)$ where $M(x) =\frac{1}{p}(f(x) -h_1(x)^{e_1}\cdots h_t(x)^{e_t})$.
    \item[(iii)] $f(x)$ does not belong to the ideal $\langle p, h_i(x) \rangle ^{2}$
in $\mathbb{Z}[x]$ for any $i$, $1 \leq i \leq t$.

\end{enumerate}
\end{Lemma}

\section{Proof of Theorem \ref{main thm}}\label{main res}

Throughout this section, let \( p \) be a prime number dividing \( D_F \).  
By Lemma \ref{jones}, the prime \( p \) divides \( [\mathcal{O}_K : \mathbb{Z}[\theta]] \) if and only if  
\( F(x) \in \langle p, g(x) \rangle^2 \) for some monic polynomial \( g(x) \in \mathbb{Z}[x] \) such that  
\( \overline{g}(x) \mid \overline{F}(x) \) and \( g(x) \) is irreducible modulo \( p \).  
It should be noted that \( F(x) \in \langle p, g(x) \rangle^2 \) only if \( \overline{g}(x) \) is a repeated factor of \( \overline{F}(x) \).

\vspace{2mm}

\noindent \textbf{Case i}.  Consider the first case, when $p$ divides $b$. We split this case into two sub-cases  according to $p$ divides $d$ or not.
\vspace{1mm}
 
\noindent \textbf{Sub-case 1.}
First, consider the case where $p \mid d$. Since $a^2 = 4d$, it follows that $p \mid a$. In this case, we have 
$F(x) \equiv (x^m + c)^n \pmod{p}$. 
Let $h(x) = x^m + c$, and denote by 
$\overline{h}(x) = \overline{h}_1^{e_1}(x)\, \overline{h}_2^{e_2}(x)\, \cdots\, \overline{h}_t^{e_t}(x)$ 
the factorization of $x^m + \overline{c}$ over $\mathbb{Z}/p\mathbb{Z}$, where each $h_i(x) \in \mathbb{Z}[x]$ 
is a distinct monic polynomial that is irreducible modulo $p$. 
We may then write 
$h(x) = h_1^{e_1}(x)\, h_2^{e_2}(x)\, \cdots\, h_t^{e_t}(x) + pH_1(x)$ 
for some polynomial $H_1(x) \in \mathbb{Z}[x]$. Note that
\begin{align*}
    F(x) =  h(x)^{n}+ah(x)^{n-1} +dh(x)^{n-2}+b 
\end{align*}
Since $n \geq 3$ and $p \mid d$, it follows that the first three terms on the right-hand side of the above expression lie in $\langle p, h_i(x) \rangle^2$ for every $i$ with $1 \leq i \leq t$, except for the last term. Hence, $F(x) \in \langle p, h_i(x) \rangle^2$ for all $i$, $1 \leq i \leq t$, if and only if $p^2 \mid b$. Consequently, by Lemma~\ref{jones}, we obtain that $p \nmid [\mathcal{O}_K : \mathbb{Z}[\theta]]$ if and only if $p^2 \nmid b$.

\vspace{1mm}

\noindent \textbf{Sub-case 2.}  $p$ does not divide $d$. We first assume that $p\neq2$. As $a^2 =4d$, we get $p\nmid a$. So, we have $\overline{F}(x) = (x^m +\Bar{c})^{n-2}[(x^m +\Bar{c})^{2} + \Bar{a}(x^m +\Bar{c}) + \Bar{d}]$.  Since $p \neq 2$, one can write $\overline{F}(x) = (x^m +\Bar{c})^{n-2}(x^m +\Bar{c} + \frac{\Bar{a}}{2})^{2}$.
Let $H(x) = x^m +c, \; G(x) = x^m +c+ \frac{a}{2}$, then $F(x) = H(x)^{n-2} G(x)^{2} \pmod p $. Let $h(x) = H(x)G(x)$  and
$
\overline{h}(x)  = \overline{h}_1(x)^{e_1}\,\overline{h}_2(x)^{e_2}\cdots \overline{h}_t(x)^{e_t}
$
be the factorization of $h(x)$ modulo $p$, where each $h_i(x) \in \mathbb{Z}[x]$ is monic, irreducible modulo $p$, and distinct. Then we can express
$
h(x) = h_1(x)^{e_1}\cdots h_t(x)^{e_t} + pH_1(x),
$
for some polynomial $H_1(x) \in \mathbb{Z}[x]$. Note that
\begin{align*}
    F(x) =  H(x)^{n-2}G(x)^2 +b.
\end{align*}
Since $n \geq 3$, the first term on the right-hand side of the above expression lies in $\langle p, h_i(x) \rangle^2$ for each $i$ with $1 \leq i \leq t$. Hence, $F(x) \in \langle p, h_i(x) \rangle^2$ for all $i$, $1 \leq i \leq t$, if and only if $p^2 \mid b$. Consequently, by Lemma~\ref{jones}, we have $p \nmid [\mathcal{O}_K : \mathbb{Z}[\theta]]$ if and only if $p^2 \nmid b$.
 If $p=2$ with $2\nmid d$, then since $a^2 =4d$, we have $2\mid a$ but $4\nmid a$. Hence, $\overline{F}(x) =  (x^m +\Bar{c})^{n-2}[(x^m +\Bar{c})^{2}  + \Bar{d}] =  (x^m +\Bar{c})^{n-2}[(x^m +\Bar{c})^{2}  + (\frac{\bar{a}}{2})^2] = (x^m +\Bar{c})^{n-2}[x^m +\Bar{c}  + \frac{\bar{a}}{2}]^2 $. So as earlier, $ 2 \nmid [\mathcal{O}_K : \mathbb{Z}[\theta]] $ if and only if $ 2^2 \nmid b $.

\vspace{2mm} 

\noindent \textbf{Case ii}. 
Assume \( p \nmid b \) and \( p \mid m \), where \( m = p^j s \) with \( j \geq 1 \).  We represent \( F(x) \) as \( F(x) \equiv ((x^s + c)^n + a(x^s + c)^{n-1} + d(x^s + c)^{n-2} + b)^{p^j} \pmod{p} \). Define \( h(x) = (x^s + c)^n + a(x^s + c)^{n-1} + d(x^s + c)^{n-2} + b \), hence \( F(x) = h(x^{p^j}) \). Let \( h(x) = h_1(x) \cdots h_t(x) \pmod{p} \) denote the factorization of \( h(x) \), where \( h_i(x) \in \mathbb{Z}[x] \) are monic, distinct, and irreducible modulo \( p \). Utilizing the binomial theorem, we have
    
\begin{align}{\label{2}}
    F(x) =&  ((x^s +c-c)^{p^j}+c)^n + a(x^m+ c)^{n-1} +d(x^m+c)^{n-2} +b  \nonumber
    \\
    =& \bigg( (x^s+c)^{p^j}+(-c)^{p^j} + \sum_{i=1}^{p^j -1}{{p^j}\choose {i}} (x^s+c)^{p^j -i} (-c)^i +c    \bigg)^n  +a(x^m+ c)^{n-1} + \nonumber
    \\ 
    & +d(x^m+c)^{n-2} +b\nonumber
    \\
=& (x^s +c)^{p^j n} +n(x^s+c)^{p^j (n-1)}t(x) + p^2N_1(x) +  a(x^m+c)^{n-1}+d(x^m+c)^{n-2} +b  \nonumber  \\
    =&(h(x) - a(x^s+c)^{n-1} -d(x^s+c)^{n-2} -b)^{p^j} +n(x^s+c)^{p^j (n-1)}t(x) + p^2N_1(x)  \nonumber
\\
&+  a(x^m+c)^{n-1} +d(x^m+c)^{n-2} +b \nonumber    \\
    =& h(x)^{p^j} + ph(x)T(x) + p^2N_1(x) + n(x^s+c)^{p^j (n-1)}t(x) \nonumber  \\
    &+ (-1)^{p^j}(a(x^s+c)^{n-1} + d (x^s+c)^{n-2} + b)^{p^j} + a(x^m+c)^{n-1} +d(x^m+c)^{n-2} +b 
\end{align}
for some $T(x), N_1(x) \in \mathbb{Z}[x]$,  where     $t(x)= \sum_{i=1}^{p^j }{{p^j}\choose {i}} (x^s+c)^{p^j -i} (-c)^i +c$. Since $j\geq 1$, first 3 terms on the right hand side of \eqref{2} belongs to  $\langle p, h_i(x)\rangle^2$ for all $i$, $1 \leq i \leq r$. Therefore, this case follows from Lemma \ref{jones}.

\medskip

\noindent \textbf{Case iii.} 
When \( p \nmid bm \) and \( p \mid d \). As $a^2=4d$, we have $p\mid a$ and $p^2\mid d$. Since \( p \mid D_F \), it follows that \( p \mid nf(c) \). First, consider the case where \( p \mid n \). Let \( n = sp^k \), where \( p \nmid s \) and \( k \geq 1 \). So 
\begin{align*}
    F(x) \equiv (x^m +c)^n +b \equiv ((x^m +c)^{s} +b)^{p^{k}} \pmod p.
\end{align*}
Let $h(x) = (x^m + c)^s + b$, and denote by $h_1(x) \cdots h_t(x)$ the factorization of $h(x)$ modulo $p$, where each $h_i(x) \in \mathbb{Z}[x]$ is a distinct monic polynomial that is irreducible modulo $p$. Since $(x^m + c)^n = (h(x) - b)^{p^k}$, we observe that 
\begin{align*}
    F(x) &= (h(x)- b)^{p^k} + a(x^m +c)^{n-1} +d(x^m +c)^{n-2}  +b \\
    & = h(x)^{p^{k}} + ph(x)H(x) +a(x^m +c)^{n-1} +d(x^m+c)^{n-2} + (-b)^{p^{k}} +b
\end{align*}
for some polynomial $H(x) \in \mathbb{Z}[x]$.
 Using $p^2 \mid d$, it follows that the first, second, and fourth terms on the right-hand side of the above equation lie in $\langle p, h_i(x) \rangle^2$ for each $i = 1, 2, \ldots, t$. 
Let $b_1 = \frac{b + (-b)^{p^k}}{p}$ and $a_1 = \frac{a}{p}$. 
Since $k \geq 1$, by Lemma~\ref{jones}, we have $p \nmid [\mathcal{O}_K : \mathbb{Z}[\theta]]$ if and only if $\overline{a}_1(x^m + \overline{c})^{n-1} + \overline{b}_1$ and $(x^m + \overline{c})^{n} + \overline{b}$ are co-prime.
 The polynomials $\overline{a}_1(x^m + \overline{c})^{n-1} + \overline{b}_1$ and $(x^m + \overline{c})^{n} + \overline{b}$ are coprime if and only if either 
(i) $p \mid a_1$ and $p \nmid b_1$, or 
(ii) $p \nmid a_1$ and the polynomials $(x^m + c)^{n-1} + \frac{\overline{b}_1}{\overline{a}_1}$ and $(x^m + c)^{n} + \overline{b}$ are coprime. 
Furthermore, if $\zeta$ is a common root of $(x^m + c)^{n-1} + \frac{\overline{b}_1}{\overline{a}_1}$ and $(x^m + c)^{n} + \overline{b}$ in the algebraic closure of $\mathbb{Z}/p\mathbb{Z}$, then 
\[
(-\overline{b})^{n-1} = ((\zeta^m + c)^n)^{n-1} = ((\zeta^m + c)^{n-1})^{n} = \left( \frac{-\overline{b}_1}{\overline{a}_1} \right)^{n}.
\]
Also, it can easily be seen that if $(\frac{-\bar{b}_1 }{\bar{a}_1})^n = (-\bar{b})^{n-1} $ then $\zeta$ defined by $\zeta^m +c = \frac{\bar{b} \bar{a}_1}{\bar{b}_1}$ is common root of polynomials   $(x^m+c)^{n-1} +\frac{\bar{b}_1}{\bar{a}_1}$ and $(x^m+c)^{n} +\bar{b}$. Therefore, (II) holds if and only if $p$ does not divide $a_1((-b_1)^n - a_1^n(-b)^{n-1})$.

Now, suppose \( p \nmid n \). Since \( p \mid D_F \), it follows that \( p \mid f(c) \), which further implies \( m \geq 2 \). Let \(\varepsilon\) be a repeated root of \(\overline{F}(x) = (x^m + \overline{c})^n + \overline{b}\) in the algebraic closure of \(\mathbb{Z}/p\mathbb{Z}\). Then \(\overline{F}(\varepsilon) = \overline{F}'(\varepsilon) = 0\). That is 
\begin{align*}
(\varepsilon^m + \overline{c})^n +\overline{b} =\overline{0}; \;\; \overline{n}(\varepsilon^m +\overline{c})^{n-1}\overline{m}\varepsilon^{m-1}=\overline{0}.
\end{align*}
Note that \(\varepsilon^m + c \neq 0\); if it were equal to zero, then \(p \mid b\), which is not true in this case. Additionally, since \(p \nmid mn\), we observe that \(\varepsilon = \overline{0}\) is the unique repeated root of \(\overline{F}(x)\) over \(\mathbb{Z}/p\mathbb{Z}\). Since \(m \geq 2\), all non-constant terms on the right-hand side of
\begin{align*}
    F(x) = \sum_{i=0}^{n-1}{{n}\choose{i}}x^{m(n-i)}c^i +a\sum_{i=0}^{n-2}{{n}\choose{i}}x^{m(n-i)}c^i+d\sum_{i=0}^{n-3}{{n}\choose{i}}x^{m(n-i)}c^i + f(c)
\end{align*}
belongs to $\langle p,x\rangle^2$.
Therefore, $F(x) \in \langle p, x \rangle^2$ if and only if $p^2 \mid f(c)$. 
Hence, by Lemma~\ref{jones}, we have $p \nmid [\mathcal{O}_K : \mathbb{Z}[\theta]]$ if and only if $p^2 \nmid f(c)$.

\vspace{2mm}
\noindent \textbf{Case iv.} 
 Consider the case when $p=2$ and $2\nmid dbm.$  Given that  $a^2 =4d$, it follows that $2\mid a$, and $4\nmid a$. 
 Since $2 \mid D_f$, and referring to Equation \eqref{disc}, we observe that $2 \mid n f(c)$. 
First suppose $2\mid n$. Write $n=2s$ and $n-2= 2(s-1)$. Clearly $2\nmid \text{gcd}(s, s-1)$. Observe that $\overline{F}(x) = (x^m+\overline{c})^n + \overline{d}(x^m +\overline{c})^{n-2} + \overline{b}$.
Let $h(x) = (x^m + c)^{s} + d(x^m + c)^{s-1} + b$. 
Suppose 
$
h(x) \equiv \prod_{i=1}^{t} h_i(x)^{r_i} \pmod{2}
$
is the factorization of $h(x)$ into a product of irreducible polynomials modulo $2$, where each $h_i(x) \in \mathbb{Z}[x]$ is monic and $r_i > 0$. 
Noting that $(h(x) - d(x^m + c)^{s-1} - b)^2 = (x^m + c)^n$, we obtain that 
 \begin{align*}
     F(x) = &h(x)^2 - 2h(x)\{d(x^m+c)^{s-1}+b\} +(d^2+d)(x^m +c)^{n-2} + a(x^m+c)^{n-1} \\ &+2bd(x^m+c)^{\frac{n-2}{2}} 
     + (b^2+b).
 \end{align*}
The first two terms on the right-hand side of the above equation lie in $\langle 2, h_i(x) \rangle^2$ for each $i = 1, 2, \ldots, t$. 
Thus, \( F(x) \in \langle 2, h_i(x) \rangle^2 \) for some \( i = 1, 2, \ldots, t \) 
if and only if the polynomials 
\[
G(x) :=  \frac{a}{2}(x^m+c)^{n-1} + \frac{d^2+d}{2}(x^m+c)^{n-2} + bd(x^m+c)^{\frac{n-2}{2}} + \frac{b^2+b}{2}  
\]
and \( h(x) \) share a common root modulo \( 2 \).
Therefore, by Lemma~\ref{jones}, 
\( F(x) \notin \langle 2, h_i(x) \rangle^2 \) for any \( i = 1, 2, \ldots, t \)
if and only if \( G(x) \) and \( h(x) \) are coprime modulo \( 2 \). We now claim that \( h(x) \) is co-prime to \( G(x) \) modulo \( 2 \)
if and only if \( b \equiv 1 \pmod{4} \).

Note that $n=2s$. Since $2\mid a$, $2\nmid bd$, and $4\nmid a$,
we can write
\(
 a=2A,d=2D+1, b=2B+1
\)
with $A,D,B\in\mathbb{Z}$ and $2\nmid A$. So
\[
h(x)\equiv (x^m+c)^{s} +(x^m+c)^{s-1}+1 \pmod 2,
\]
and $G(x)$ can be written as
\[
G(x)=A(x^m+c)^{2s-1}+(2D^2+3D+1)(x^m+c)^{2s-2}+ (2D+1)(2B+1)(x^m+c)^{s-1}+  (2B^2+3B+1),
\]
so that
\[
G(x)\equiv (x^m+c)^{2s-1} + (D + 1)(x^m+c)^{2s-2} +(x^m+c)^{s-1} +B+1    \pmod{2}.
\]
Setting $\delta \equiv D\pmod{2}$ and $\gamma\equiv B\pmod{2}$ (hence $\delta,\gamma\in\{0,1\}$). Thus, it remains to establish the claim in the following four cases. 

\[
\text{(i) }\delta=0,\ \gamma=0,\qquad
\text{(ii) }\delta=0,\ \gamma=1,\qquad
\text{(iii) }\delta=1,\ \gamma=0,\qquad
\text{(iv) }\delta=1,\ \gamma=1.
\]
It is easy to observe that $G(x)$ is co-prime to $h(x)$ modulo $2$ in (i), (iii) and (iv), and fails to be coprime only in (ii), i.e. precisely when $\delta=0,\gamma=1$.

Returning to the definitions of $\delta$ and $\gamma$, this occurs precisely when
\[
D\equiv 0\pmod{2}\quad\text{and}\quad B\equiv 1\pmod{2},
\]
which is equivalent to
\[
d\equiv 1\pmod{4}\quad\text{and}\quad b\equiv 3\pmod{4}.
\]
By negating this condition, we get $d\equiv 3\pmod{4} \text{ or } b\equiv 1\pmod{4}.$ Since $d=(a/2)^2$, so $d\not\equiv 3\pmod{4}$. Hence, we get $b\equiv 1\pmod{4}$.

\noindent Now suppose $2 \nmid n$. Then $2 \mid f(c)$, which leads to a contradiction.  Assume $2 \mid f(c)$. Let $\varepsilon$ be a repeated root of the polynomial
$
F(x) = (x^m + c)^n + a(x^m + c)^{n-1} + d(x^m + c)^{n-2} + b \pmod 2
$ in the algebraic closure of $\mathbb{Z}/2\mathbb{Z}$. Then
\begin{align*}
    F(\varepsilon) &\equiv (\varepsilon^m + c)^{n-2} ((\varepsilon^m +c)^2+d)+b \equiv 0 \pmod 2,
    \\
    F^\prime(\varepsilon)  &\equiv n(\varepsilon^m +c)^{n-3}m\varepsilon^{m-1} ((\varepsilon^m +c)^2+d)  \equiv 0 \pmod2.
\end{align*}
Since $2 \nmid nm$, it follows that either $\varepsilon \equiv 0 \pmod{2}$, or $\varepsilon^m + c \equiv 0 \pmod{2}$, or $(\varepsilon^m + c)^2 + d \equiv 0 \pmod{2}$. If $\varepsilon \equiv 0 \pmod{2}$, then $2 \mid c^n + dc^{\,n-2} + b$, i.e,  $2 \mid c^{n-2}(c^2+ d) + b$, and whether $2 \mid c$ or $2 \nmid c$, in both cases $2 \mid b$, a contradiction. The cases $\varepsilon^m + c \equiv 0 \pmod{2}$ and  $(\varepsilon^m + c)^2 + d \equiv 0 \pmod{2}$ are also not possible since $2 \nmid b$.  Therefore, the assumption $2 \mid f(c)$ cannot hold.

\vspace{2mm}

\noindent \textbf{Case v.} Consider the case when $p \nmid abm$. In view of $a^2=4d$, we have  $p\neq2$. Let $\varepsilon$ be a repeated root of $F(x)$ in the algebraic closure of $\mathbb{Z}/p\mathbb{Z}$. That is 
\begin{align*}
F(\varepsilon) &= (\varepsilon^m + c)^n +a(\varepsilon^m +c)^{n-1} +d(\varepsilon^m +c)^{n-2} + b \equiv0 \pmod p ,
\\
F^\prime (\varepsilon) &= m\varepsilon^{m-1}(\varepsilon^m +c)^{n-3}\{ n(\varepsilon^m +c)^{2} + a(n-1)(\varepsilon^m +c) + (a/2)^2 (n-2) \} \equiv0 \pmod p .
\end{align*}

\noindent
Note that $\varepsilon^m +c \equiv 0 \pmod p \implies p\mid b$, which is not true, so $\varepsilon^m +c \not\equiv 0 \pmod p$.  Observe that \( \varepsilon = \overline{0} \) is a root if and only if \( p \mid f(c) \). In this case, \( F(x) \notin \langle p, x \rangle^2 \) if and only if \( p^2 \nmid f(c) \). If $\varepsilon \neq \overline{0}$ then using second congruence above we obtain  
\begin{align*}
     n(\varepsilon^m +c)^{2} + a(n-1)(\varepsilon^m +c) + (a/2)^2 (n-2)  \equiv0 \pmod p.
\end{align*}

\noindent Observe that $p\nmid n$. For if $n \equiv 0\pmod p$ then we would have $\varepsilon^m +c = \frac{-a}{2} \pmod p$. Substituting this into the congruence $F(\varepsilon)\equiv 0\pmod p $ implies $p\mid b$, a contradiction. In light of the expression of $\bar{F}^\prime (\varepsilon)$ above and also that $p\nmid bn$, $ \varepsilon^m +\overline{c} = \frac{-\overline{a}(\overline{n-2})}{\overline{2n}}$  or $\varepsilon^m +\overline{c} = \frac{-\overline{a}}{\overline{2}}.$ But putting $\varepsilon^m +\overline{c} = \frac{-\overline{a}}{\overline{2}}$ in the expression for $\overline{F}(\varepsilon)$ above implies $p\nmid b$. a contradiction. Therefore, we have 
\begin{align}{\label{5}}
    \varepsilon^m +\overline{c} = \frac{-\overline{a}(\overline{n-2})}{\overline{2n}}.
\end{align}
Also, since $F^{\prime \prime}(\varepsilon) \not\equiv 0  \pmod p $, it has multiplicity 2.
Select $z\in \mathbb{Z}$ in such a manner that
\begin{align}{\label{6}}
    2nz\equiv -a(n-2) (\operatorname{mod} p^2)  .
\end{align}
Therefore, we have 
$$
\bar{F}(x) = (x^m + \bar{c}-\bar{z})^2\bar{h}(x),
$$
where $h(x) \in \mathbb{Z}/p\mathbb{Z}[x]$ is a separable polynomial. It is evident that \( F(x) \) can be expressed as
\begin{align}{\label{7}}
    F(x) = (x^m + c-z)q(x) + f(z),
\end{align}
where $$q(x) = \sum_{i=1}^n(x^m+c)^{n-i}z^{i-1} + a\sum_{i=2}^n(x^m+c)^{n-i}z^{i-2} + d\sum_{i=3}^n(x^m+c)^{n-i}z^{i-3}.$$ 
Observe that $\bar{q}(x) = (x^m +\bar{c}-\bar{z})\bar{h}(x)$ and $x^m +\bar{c}-\bar{z}$ does not divide $\bar{h}(x)$. Let \( \overline{h}(x) = \overline{h}_1(x) \dots \overline{h}_r(x) \) be the factorization of $\bar{h}(x)$, where \( \overline{h}_i(x) \) are distinct monic irreducible polynomials over \( \mathbb{Z}/p\mathbb{Z} \). Write
\begin{align*}
    q(x) = (x^m +c -z)\prod_{i=1}^rh_i(x) + pg_1(x)
\end{align*}
for some $g_1(x) \in \mathbb{Z}[x]$. Substituting the above equation in eq. \eqref{7}, we see that
\begin{align*}
     F(x) = (x^m +c -z)^2\prod_{i=1}^rh_i(x) + p(x^m +c -z)g_1(x) + f(z).
\end{align*}
Therefore using Lemma \ref{jones}, we see that $F(x) \notin \langle p, x^m +c -z \rangle ^2 \bigcup \langle p,h_i(x)\rangle^2$ if and only if $p^2 \nmid f(z)$. We now claim that  $p^2 \mid f(z)$ if and only if $p^2 \mid  (-b)^{n-2}\{(-n)^nb +4 (n-2)^{n-2}(a/2)^n\}$.

\noindent Using eq. \ref{6} and that $a^2=4d$ we have 
\begin{align}{\label{8}}
    (2n)^n f(z) \equiv (-1)^n[4a^n(n-2)^{n-2}+b(-2n)^n] (\text{mod } p^2).
\end{align}
Since $p\nmid 2b$ and multiplying $b^{n-2}$ to above equation on both sides we see that
$$
b^{n-2}n^n f(z) \equiv (-b)^{n-2}\{(-n)^nb +4 (n-2)^{n-2}(a/2)^n\} (\text{mod } p^2).
$$
 As $p \nmid nb$, claim follows.  Thus, according to Lemma \ref{jones}, we deduce, by combining both parts when $\varepsilon=\overline{0}$ and when $\varepsilon\neq \overline{0}$, that \( p \nmid [\mathcal{O}_K : Z(\theta)] \) if and only if \( p^2 \nmid  f(c) \) and  \( p^2 \nmid  (-b)^{n-2}\{(-n)^nb +4 (n-2)^{n-2}(a/2)^n\} \).
\qed

\section{Proof of Theorem \ref{t2}}

 The given differential equation is \begin{align}{\label{diff1}}
    \bigg( \frac{d^m }{dx^m} +c \bigg )^n y+ a\bigg( \frac{d^m }{dx^m} +c \bigg )^{n-1} y+d\bigg( \frac{d^m }{dx^m} +c \bigg )^
    {n-2}y+by=0, 
    \end{align}
    where $a^2=4d,\; n\geq3,\; m\geq 2$.  Observe that $\mathcal{F}(z) = (z^m + c)^n + a(z^m + c)^{n-1} +(z^m + c)^{n-2} + b$ is the irreducible auxiliary equation associated with \eqref{diff1}, having root $\theta$ and splitting field $K(\theta)$.  Suppose that every prime $p$ dividing the discriminant $D_{\mathcal{F}}$ satisfies one of the conditions $(i)–(v)$ of Theorem \ref{main thm}. Then, by the formula $
D_{\mathcal{F}} = [\mathcal{O}_K : \mathbb{Z}[\theta]]^2 d_K,
$ it follows that $\mathcal{O}_K = \mathbb{Z}[\theta]$, where $\mathcal{O}_K$ denotes the ring of integers of the number field $K = \mathbb{Q}(\theta)$. Also
 $$
 \mathbb{Z}[\theta] = \{c_0 + c_1\theta + c_2 \theta^2 +\cdots + c_{mn-1}\theta^{mn-1} \;|\; c_k \in \mathbb{Z} \text{ for all } 1\leq k\leq mn-1     \}.
 $$
Thus, every root of the equation $\mathcal{F}(z) = 0$ can be expressed in the form
$$
c_0^{(i)} + c_1^{(i)}\theta + c_2^{(i)}\theta^2 + \cdots + c_{mn-1}^{(i)}\theta^{mn-1},
$$
where each $c_{j-1}^{(i)} \in \mathbb{Z}$ for all $1 \leq i,j \leq mn$. Consequently, the general solution of the differential equation \eqref{diff1} is given by
\begin{align*}
    y(x) = \sum_{i=1}^{mn} \alpha_i\prod_{j=1}^{mn}e^{c_{j-1}^{(i)} \theta^{j-1}x},
\end{align*}
where  $\alpha_i$ is arbitrary real constants for all $1\leq i,j \leq mn$. \qed

\medskip

\section{Proof of Corollary \ref{analytic c} and Theorem \ref{analytic t}}

\vspace{2mm}

\noindent
\begin{proof}[Proof of Corollary \ref{analytic c}.]
Let \(p\) be a prime divisor of \(D_F\). First, suppose that \(p\mid b\). Since \(b\) is square-free, we have
$
p^2\nmid b.
$
Therefore condition \((i)\) of Theorem \ref{main thm} holds. Next, suppose that \(p\nmid b\) and \(p\mid m\). Since \(\operatorname{rad}(m)\mid b\),
every prime divisor of \(m\) is also a divisor of \(b\). This contradicts \(p\nmid b\).
Hence this case does not arise under the assumptions.

\vspace{2mm} 

Now suppose that $ p\nmid bm$ and $p\mid d.$
Since \(a^2=4d\), we get \(p\mid a\). We consider two subcases. If \(p\nmid n\) then the assumption $ \nu_p(f(c))\leq 1,$
yields $p^2\nmid f(c).$
Therefore condition \((iii)(b)\) of Theorem \ref{main thm} holds. 
Now suppose that \(p\mid n\). Write
$
n=sp^k,\; k\geq 1,\; p\nmid s.
$
Set
$
a_1=\frac{a}{p},
\;
b_1=\frac{b+(-b)^{p^k}}{p}.
$
We verify condition \((iii)(a)\) of Theorem \ref{main thm}. If \(p^2\mid a\), then \(p\mid a_1\). Then the assumption,
$ (-b)^p\not\equiv -b \pmod {p^2}, $ and \(p\nmid b\), gives $b+(-b)^{p^k}\not\equiv 0 \pmod {p^2}. $ Hence, $ p\nmid b_1.$ 
Thus $p\mid a_1 \text{ and } p\nmid b_1.$
So the first alternative of condition \((iii)(a)\) holds. If \(p\Vert a\), then \(p\nmid a_1\). By assumption, $ (-b)^p\equiv -b \pmod {p^2}$ and \(p\nmid b\), we get $ (-b)^{p^k}\equiv -b \pmod {p^2}.$ 
Therefore, $p\mid b_1.$
Hence, $ (-b_1)^n\equiv 0 \pmod p. $
Also, since \(p\nmid a_1\) and \(p\nmid b\), we have $ a_1^n(-b)^{n-1}\not\equiv 0 \pmod p.$
Thus, we have $ (-b_1)^n-a_1^n(-b)^{n-1}\not\equiv 0 \pmod p.$ 
Since \(p\nmid a_1\), we get
$
p\nmid a_1\left((-b_1)^n-a_1^n(-b)^{n-1}\right).
$
Therefore the second alternative of condition \((iii)(a)\) holds.

\vspace{1mm}

Next, suppose that
$
p=2
\text{ and }
2\nmid dbm.
$
By assumption,
$
b\equiv 1 \pmod 4.
$
Hence condition \((iv)\) of Theorem \ref{main thm} holds.

\vspace{1mm}

Finally, suppose that
$
p\nmid abm.
$
Then \(p\nmid b\). In this case, the assumption
$$
\nu_p(f(c))\leq 1
\text{ and } 
\nu_p\left((-n)^n b+4(n-2)^{n-2}\left(\frac a2\right)^n\right)\leq 1,
$$
implies
$p^2\nmid f(c) $
and
$p^2\nmid
\left\{(-n)^n b+4(n-2)^{n-2}\left(\frac a2\right)^n\right\}.$
Thus condition \((v)\) of Theorem \ref{main thm} holds.

Hence every prime divisor \(p\) of \(D_F\) satisfies one of the conditions
\((i)\)--\((v)\) of Theorem \ref{main thm}. Therefore, by Theorem \ref{main thm}, no prime divisor of
\(D_F\) divides the index
$
[\mathcal{O}_K:\mathbb{Z}[\theta]].$ 
Since every prime divisor of the index must divide the discriminant \(D_F\), it follows
that
$
\mathcal{O}_K=\mathbb{Z}[\theta].
$
Hence \(F(x)\) is monogenic.

\vspace{2mm}

Now we prove that \(f(x)\) is also monogenic. Let \(\alpha\) be a root of \(f(x)\), and put
\(L=\mathbb{Q}(\alpha)\). The discriminant of \(f(x)\) is
\[
D_f
=
(-1)^{n(n-1)/2}
(-b)^{n-2}
\left\{(-n)^n b+4(n-2)^{n-2}\left(\frac a2\right)^n\right\}.
\]
Let \(p\) be a prime divisor of \(D_f\).
If \(p\mid b\), then \(p^2\nmid b\), because \(b\) is square-free. Hence \(p \nmid [\mathcal{O}_L:\mathbb{Z}[\alpha]].\)
If \(p\nmid b\), then by assumption
$
\nu_p\left((-n)^n b+4(n-2)^{n-2}\left(\frac a2\right)^n\right)\leq 1
$,  we have
$p^2$ does not divide $(-n)^n b+4(n-2)^{n-2}\left(\frac a2\right)^n$ . Hence \(p \nmid [\mathcal{O}_L:\mathbb{Z}[\alpha]].\) 
For primes \(p\mid \gcd(d,n)\), the congruence assumptions in condition \((3)\) of the corollary gives the required alternatives condition \((iii)(a)\) which can be proved exactly using the arguments above. Therefore no prime divisor of \(D_f\) divides $[\mathcal{O}_L:\mathbb{Z}[\alpha]].$ Hence,
$[\mathcal{O}_L:\mathbb{Z}[\alpha]]=1. $ Therefore, \(f(x)\) is also monogenic. 
\end{proof}

The following theorem proved by Granville \cite[Theorem 1]{granville}  is the key ingredient in our proof of Theorem \ref{analytic t}.

\begin{Theorem}{\cite[Theorem 1]{granville}}{\label{gran}}
    Suppose that \(f(x)\in \mathbb{Z}[x]\), without any repeated roots. Let \(D=\gcd\{f(n):n\in \mathbb{Z}\}\), and select \(D'\) to be the smallest divisor of \(D\) for which \(D/D'\) is square-free. If the abc-conjecture is true, then
\[
\#\{1\le n\le X:f(n)/D' \text{ is square-free}\}\sim \lambda_f X
\]
where \(\lambda_f>0\) is a positive constant, which we determine as follows:
\[
\lambda_f=\prod_{p\text{ prime}}\left(1-\frac{\omega_f(p)}{p^{2+q_p}}\right),
\]
where for each prime $p$, we let $q_p$ be the largest power of $p$ which divides $D'$ and let $\omega_f(p)$ denote the number of integers $a$ in the range $1\le a\le p^{2+q_p}$ for which
$
f(a)\equiv 0 \pmod{p^2}.
$
\end{Theorem}

\begin{proof}[Proof of Theorem \ref{analytic t}.]
For a square-free integer $b$, define
\[
H_b(x)=f(\ell x)=(\ell x)^n+a(\ell x)^{n-1}+d(\ell x)^{n-2}+b,
\]
and let
$
D_b=\gcd\{H_b(r):r\in \mathbb Z\}.
$
Since $D_b\mid H_b(0)=b$ and $b$ is square-free, the integer $D_b$ is square-free.
Therefore, in the notation of Theorem \ref{gran}, we have 
 $D_b'=1$. Hence, assuming the
abc-conjecture, Theorem \ref{gran} gives
\[
\#\{1\le r\le C/\ell:H_b(r)\ \text{is square-free}\}
\sim \lambda_b\frac{C}{\ell},
\]
where
\[
\lambda_b=
\prod_{p\ \mathrm{prime}}
\left(1-\frac{\omega_{H_b}(p)}{p^2}\right),
\]
and $\omega_{H_b}(p)$ denotes the number of integers $\alpha$, with
$1\le \alpha\le p^2$, such that
$
H_b(\alpha)\equiv 0 \pmod{p^2}.
$
Observe that $ H_b(p^2)\equiv b \pmod{p^2} $ which is clearly non-zero modulo $p^2$. Hence, we have $\omega_{H_b}(p)\le p^2-1. $
Also note that $\deg(H_b)=n$, so $\omega_{H_b}(p) $ cannot exceed $n$.
Thus, we have
$
\omega_{H_b}(p)\le \min\{n,p^2-1\}.
$
Using these observations, it follows that
\[
\lambda_b\ge
\prod_{p\le\sqrt n}\frac{1}{p^2}
\prod_{p>\sqrt n}\left(1-\frac{n}{p^2}\right)
=\Lambda_n.
\]
Thus, for every fixed square-free integer $b$, the number of integers $c\le C$ such that
$\ell\mid c$ and $f(c)$ is square-free is at least
$
\left(\Lambda_n+o(1)\right)\frac{C}{\ell}.
$

Recall that $G(y) = (-n)^n y + 4(n-2)^{n-2} \left( \frac{a}{2} \right)^n$. Since $G$ is linear, it has no repeated root. Therefore, assuming the abc-conjecture,
Theorem \ref{gran} yields
\[
\#\{1\le b\le B:G(b)/D_G'\ \text{is square-free}\}
\sim \lambda_G B.
\]

Among these integers $b$, consider those which are square-free, divisible by $\kappa$,
and satisfy the congruence conditions required in $(3)$ of Corollary \ref{analytic c}. For each prime
$p\mid \rho$, choose $\gamma_p\in (\mathbb Z/p^2\mathbb Z)^\times$ such that
\[
\gamma_p^p\not\equiv \gamma_p \pmod{p^2}
\quad\text{if }p^2\mid a,
\]
and
\[
\gamma_p^p\equiv \gamma_p \pmod{p^2}
\quad\text{if }p\parallel a.
\]
By the Chinese remainder theorem, there exists $\gamma \in (\mathbb{Z}/\rho^2\mathbb{Z})^\times$
such that
\[
\gamma \equiv -\gamma_p \pmod{p^2}
\qquad \text{for all } p\mid \rho.
\]
We require that
$
b\equiv \gamma \pmod{\rho^2}.
$
If $2\nmid dm$, we also require that
$
b\equiv 1 \pmod{4}.
$
Since $\gcd(\kappa,\rho)=1$, these conditions can be satisfied simultaneously. The square-free density
for multiples of the square-free integer $\kappa$ contributes the factor
$
\frac{\varphi(\kappa)}{\kappa^2\zeta(2)},
$
the congruence condition modulo $\rho^2$ contributes the factor $1/\rho^2$, and the
additional congruence modulo $4$ contributes the factor $1/\delta$. Hence the number of integers $b\le B$ satisfying the above conditions is at least
\[
\left(
\lambda_G\frac{\varphi(\kappa)}{\delta\,\kappa^2\zeta(2)\rho^2}
+o(1)
\right)B.
\]

Let $b$ and $c$ be chosen as above. Since $\ell\mid\kappa\mid b$ and $b$ is square-free,
we have
$
\nu_\ell(b)=1.
$
Also, from $a^2=4d$ and $\ell\mid d$, it follows that $\ell\mid a$. Hence every
non-leading coefficient of
$
f(x)=x^n+ax^{n-1}+dx^{n-2}+b
$
is divisible by $\ell$, whereas the constant term is not divisible by $\ell^2$.
Therefore $f(x)$ is $\ell$-Eisenstein, and hence irreducible. Now consider
$
F(x)=(x^m+c)^n+a(x^m+c)^{n-1}+d(x^m+c)^{n-2}+b.
$
Since $\ell\mid c$, every non-leading coefficient of $F(x)$ is divisible by $\ell$.
Moreover,
$
F(0)=c^n+ac^{n-1}+dc^{n-2}+b=f(c).
$
Because $n\ge 3$ and $\ell$ divides each of $ a,d,c$, the first three terms in $f(c)$ are divisible by
$\ell^2$, and therefore
$
f(c)\equiv b \pmod{\ell^2}.
$
Since $\nu_\ell(b)=1$, it follows that
$
\nu_\ell(f(c))=1.
$
Thus $F(x)$ is also $\ell$-Eisenstein, and hence irreducible.

It remains to verify the hypotheses of Corollary \ref{analytic c}. By construction, $b$ is square-free
and $f(c)$ is square-free. Moreover, for every integer $b\le B$ satisfying the above conditions, we have
\[
\nu_p\!\left((-n)^n b+4(n-2)^{n-2}\left(\frac{a}{2}\right)^n\right)\le 1.
\]
Hence condition (1) of Corollary \ref{analytic c} holds. If $2\nmid dm$, then
$
b\equiv 1\pmod 4,
$
so condition (2) of Corollary \ref{analytic c} holds. Finally, if $p\mid\gcd(d,n)$, then
$
b\equiv \gamma\equiv -\gamma_p \pmod{p^2},
$
and hence
$
-b\equiv \gamma_p \pmod{p^2}.
$
Therefore
\[
(-b)^p\not\equiv -b \pmod{p^2}
\quad\text{when }p^2\mid a,
\]
and
\[
(-b)^p\equiv -b \pmod{p^2}
\quad\text{when }p\parallel a.
\]
Thus condition (3) of Corollary \ref{analytic c} also holds. Consequently, both $f(x)$ and $F(x)$
are monogenic.

Combining the lower bound for admissible values of $b$ with the lower bound for
admissible values of $c$ corresponding to each such $b$, we conclude that the number
of pairs $(b,c)$ with $1\le b\le B$ and $1\le c\le C$ for which both
\[
f(x)=x^n+ax^{n-1}+dx^{n-2}+b
\quad\text{and}\quad
F(x)=(x^m+c)^n+a(x^m+c)^{n-1}+d(x^m+c)^{n-2}+b
\]
are irreducible and monogenic is at least
\[
\Lambda_n\lambda_G
\frac{\varphi(\kappa)}{\delta\,\ell\,\kappa^2\zeta(2)\rho^2}\,BC+o(BC).
\]
This proves the theorem.
\end{proof}

\section*{Declarations}
\subsection*{Data Availability} This manuscript has no associated data.
\subsection*{Conflict of Interest}  The authors declare that they have no competing interests.
\subsection*{Authors Contribution} All authors have equal contribution.


\bibliographystyle{abbrv}
			\bibliography{bib}

@article {Erdos53,
    AUTHOR = {Erd\"os, P.},
     TITLE = {Arithmetical properties of polynomials},
   JOURNAL = {J. London Math. Soc.},
  FJOURNAL = {The Journal of the London Mathematical Society},
    VOLUME = {28},
      YEAR = {1953},
     PAGES = {416--425},
      ISSN = {0024-6107,1469-7750},
   MRCLASS = {10.0X},
  MRNUMBER = {56635},
MRREVIEWER = {L.\ Carlitz},
       DOI = {10.1112/jlms/s1-28.4.416},
       URL = {https://doi.org/10.1112/jlms/s1-28.4.416},
}

@article{jhorar20166,
  title={ When is R[$\theta$] integrally closed?},
  author={Khanduja, Sudesh K and Jhorar, Bablesh},
  journal={Journal of Algebra and Its Applications},
  volume={15},
  number={05},
  pages={1650091},
  year={2016},
  publisher={World Scientific}
}

@article{jones_taiwan,
  title={Monogenic binomial compositions},
  author={Harrington, Joshua and Jones, Lenny},
  journal={Taiwanese Journal of Mathematics},
  volume={24},
  number={5},
  pages={1073--1090},
  year={2020},
  publisher={JSTOR}
}

@article{gaal2022,
  title={On the monogenity of certain binomial compositions},
  author={Ga{\'a}l, Istv{\'a}n},
  journal={JP Journal of Algebra, Number Theory and Applications},
  volume={57},
  pages={1--16},
  year={2022}
}

@article{hasse,
    title={Zahlentheorie},
  author={H. Hasse},
journal={Akademie-Verlag},
year={Berlin, 1963}
}

@article{Gaal_book,
    title={Diophantine equations and power integral bases: Theory and Algorithms, 2nd Edi.},
  author={I. Ga\'{a}l},
journal={Birkh\"{a}auser/Springer},
year={Cham (2019)}
}

@book{h_cohen,
  title={A course in computational algebraic number theory},
  author={Cohen, Henri},
  volume={138},
  year={2013},
  publisher={Springer Science \& Business Media}
}

@article{dedekind,
  title={{\"U}ber den Zusammenhang zwischen der Theorie der Ideale und der Theorie der h{\"o}heren Congruenzen},
  author={Dedekind, Richard},
  year={1878},
  publisher={Dieterich'sche Verlags-Buchhandlung}
}

@article{jakhar_jnt,
  title={On prime divisors of the index of an algebraic integer},
  author={Jakhar, Anuj and Khanduja, Sudesh K and Sangwan, Neeraj},
  journal={Journal of Number Theory},
  volume={166},
  pages={47--61},
  year={2016},
  publisher={Elsevier}
}

@misc{sage,
  title={The Sage Mathematics Software System (Version 9.0)},
  author={Sage Developers, Sagemath},
  year={2020}
}

@article {uchida,
    AUTHOR = {Uchida, K\^oji},
     TITLE = {When is {$Z[\alpha ]$} the ring of the integers?},
   JOURNAL = {Osaka Math. J.},
  FJOURNAL = {Osaka Mathematical Journal},
    VOLUME = {14},
      YEAR = {1977},
    NUMBER = {1},
     PAGES = {155--157},
      ISSN = {0388-0699},
   MRCLASS = {13F05},
  MRNUMBER = {450255},
MRREVIEWER = {E.\ H.\ Batho},
       URL = {http://projecteuclid.org/euclid.ojm/1200770215},
}

@article{jc,
  title={https://studylib.net/doc/8187082/the-discriminant-of-a-composition-of-two},
  author={Cullinan, John},
journal={Department of Mathematics, Bard College, Annandale-On-Hudson, NY 12504},
  year={}
}

@article{jakharrocky,
  title={On primes dividing the index of a quadrinomial},
  author={Jakhar, Anuj},
journal={Rocky Mountain J. Math. },
volume={60},
  number={6},
  pages={2117-2125},
  year={2020}
}

@article{himanshu_ramanujan,
  title={Monogenity of composition of polynomials},
  author={Sharma, Himanshu and Sarma, Ritumoni},
  journal={The Ramanujan Journal},
  volume={67},
  number={2},
  pages={45},
  year={2025},
  publisher={Springer}
}

@article {suman_ffa,
    AUTHOR = {Kaur, Sumandeep and Kumar, Surender and Remete, L\'aszl\'o},
     TITLE = {On the index of power compositional polynomials},
   JOURNAL = {Finite Fields Appl.},
  FJOURNAL = {Finite Fields and their Applications},
    VOLUME = {107},
      YEAR = {2025},
     PAGES = {102642 (20 pages)},
      ISSN = {1071-5797,1090-2465},
   MRCLASS = {11R04 (11R09)},
  MRNUMBER = {4902688},
       DOI = {10.1016/j.ffa.2025.102642},
       URL = {https://doi.org/10.1016/j.ffa.2025.102642},
}

@article{granville,
  title={ABC allows us to count squarefrees.},
  author={Granville, Andrew},
  journal={IMRN: International Mathematics Research Notices},
  volume={1998},
  number={19},
  year={1998}
}

@article{gaalsurvey,
  title={Monogenity and power integral bases: recent developments},
  author={Ga{\'a}l, Istv{\'a}n},
  journal={Axioms},
  volume={13},
  number={7},
  pages={429},
  year={2024},
  publisher={MDPI}
}

@inproceedings{jonespower,
  title={The monogenity of power-compositional Eisenstein polynomials},
  author={Jones, Lenny},
  booktitle={Annales Mathematicae et Informaticae},
  volume={55},
  pages={93--113},
  year={2022},
  organization={Eszterh{\'a}zy K{\'a}roly Egyetem L{\'\i}ceum Kiad{\'o}}
}

@article{harrington2022,
  title={The irreducibility and monogenicity of power-compositional trinomials},
  author={Harrington, Joshua and Jones, Lenny},
  journal={arXiv preprint arXiv:2204.07784},
  year={2022}
}

@article{jones2024monogenicity,
  title={The monogenicity of power-compositional characteristic polynomials},
  author={Jones, Lenny},
  journal={Albanian Journal of Mathematics},
  volume={18},
  number={1},
  pages={21--30},
  year={2024},
  publisher={Research Institute of Science and Technology (RISAT)}
}

@article{jakhar2025study,
  title={A study of monogenity of binomial composition.},
  author={Jakhar, Anuj and Kalwaniya, Ravi and Yadav, Prabhakar},
  journal={Acta Arithmetica},
  volume={221},
  number={4},
  year={2025}
}

@article{kedlaya,
  title={A construction of polynomials with squarefree discriminants},
  author={Kedlaya, Kiran},
  journal={Proceedings of the American Mathematical Society},
  volume={140},
  number={9},
  pages={3025--3033},
  year={2012}
}

@article{JonesNonSquarefree,
  title={Monogenic polynomials with non-squarefree discriminant},
  author={Jones, Lenny},
  journal={Proceedings of the American Mathematical Society},
  volume={148},
  number={4},
  pages={1527--1533},
  year={2020}
}

@article{BarmanNarodeWagh,
  title={On the Monogenity of Polynomials with Non-Squarefree Discriminants},
  author={Barman, Rupam and Narode, Anuj and Wagh, Vinay},
  journal={arXiv preprint arXiv:2506.16496},
  year={2025}
}

\end{document}